\documentclass[12pt, reqno, a4paper]{amsart}
\usepackage{amsfonts, amsmath, amsthm, amstext, amssymb}
\usepackage{mathtools}
\usepackage{url}
\usepackage{paralist}
\usepackage{tikz}\usetikzlibrary{arrows}
\usepackage{tikz-cd}
\usepackage[english]{babel}
\usepackage{framed}
\usepackage{amsrefs}
\usepackage[marginratio=1:1,height=660pt,width=410pt,tmargin=100pt]{geometry}
\usepackage{lipsum}
\usepackage{layout}
\usepackage{hyperref}

\makeatletter
\@namedef{subjclassname@2020}{\textup{2020} Mathematics Subject Classification}
\makeatother

\DeclareFontFamily{U}{mathx}{\hyphenchar\font45}
\DeclareFontShape{U}{mathx}{m}{n}{
      <5> <6> <7> <8> <9> <10>
      <10.95> <12> <14.4> <17.28> <20.74> <24.88>
      mathx10
      }{}
\DeclareSymbolFont{mathx}{U}{mathx}{m}{n}
\DeclareFontSubstitution{U}{mathx}{m}{n}
\DeclareMathAccent{\widecheck}{0}{mathx}{"71}

\usepackage[OT2,T1]{fontenc}
\DeclareSymbolFont{cyrletters}{OT2}{wncyr}{m}{n}
\DeclareMathSymbol{\Sha}{\mathalpha}{cyrletters}{"58}

\DeclareMathOperator{\tbG} {\bf{G}}

\DeclareMathOperator{\tbMT} {\mc{MT}}

\DeclareMathOperator{\trd} {tr.deg.}
\DeclareMathOperator{\rank} {rank}

\DeclareMathOperator{\Zar} {Zar}
\DeclareMathOperator{\id} {id}
\DeclareMathOperator{\nd} {nd}

\DeclareMathOperator{\Aut} {Aut}

\DeclareMathOperator{\Gal} {Gal}
\DeclareMathOperator{\Hom} {Hom}

\DeclareMathOperator{\defin} {def}

\DeclareMathOperator{\CC} {\mathbb{C}}

\DeclareMathOperator{\RR} {\mathbb{R}}

\DeclareMathOperator{\QQ} {\mathbb{Q}}

\DeclareMathOperator{\ZZ} {\mathbb{Z}}

\newtheorem{lemma}{Lemma}[section]
\newtheorem{thm}[lemma]{Theorem}
\newtheorem{cor}[lemma]{Corollary}
\theoremstyle{definition}\newtheorem{definition}[lemma]{Definition}
\theoremstyle{definition}
\theoremstyle{definition}

\newcommand{\mb}[1]{\mathbb{#1}}
\newcommand{\mc}[1]{\mathcal{#1}}
\newcommand{\mf}[1]{\mathfrak{#1}}
\newcommand{\ra}{\rightarrow}
\newcommand{\ol}[1]{\overline{#1}}
\newcommand{\wt}[1]{\widetilde{#1}}
\newcommand{\tb}[1]{\textbf{#1}}
\newcommand{\bs}{\backslash}
\newcommand{\wc}{\widecheck}

\begin{document}
\title[Mixed Ax-Schanuel with derivatives]{Ax-Schanuel with derivatives for mixed period mappings}
\author{Kenneth Chung Tak Chiu}
\address{Department of Mathematics,  University of Toronto, Toronto, Canada.}
\email{kennethct.chiu@alumni.utoronto.ca}
\begin{abstract}
We prove the Ax-Schanuel property of the derivatives of mixed period mappings. We also prove the jet space reformulation of this result. The proofs use the Ax-Schanuel result for principal bundles with flat connections obtained by Bl\'{a}zquez-Sanz, Casale,  Freitag, and Nagloo.
\end{abstract}
\maketitle

\section{Introduction}
\subsection{Motivation}
In 1971, Ax proved the function field analogue  \cite{Ax}  of the Schanuel conjecture for exponentials. 
Ax also generalized the result to exponential maps of semi-abelian varieties a year after \cite{Ax72}.
These were later extended to other functions in variational Hodge theory, e.g. the $j$-function by Pila-Tsimerman  \cite{PT}, uniformizations of Shimura varieties by Mok-Pila-Tsimerman \cite{MPT}, period mappings by Bakker-Tsimerman  \cite{BT}, and mixed period mappings by Gao-Klingler \cite{GK} and the author \cite{C} independently.  
In Pila-Tsimerman \cite{PT} and Mok-Pila-Tsimerman \cite{MPT}, Ax Schanuel theorems for derivatives of respectively the $j$-function and uniformizations of Shimura varieties were also included.  In this paper, we generalize the above results by proving the Ax-Schanuel property of the derivatives of mixed period mappings. 

These Ax-Schanuel results have applications on existential closedness problems. For instance, Ax-Schanuel for derivatives of $j$-functions by Pila-Tsimerman \cite{PT} was applied in the work of Aslanyan, Eterovi\'{c}, and Kirby \cite{AEK} on the existential closedness problem for the $j$-function; 
while Ax-Schanuel for derivatives of uniformizations of Shimura varieties  \cite{MPT} is applied in the work of Eterovi\'{c} and Zhao \cite{EZ} on the same problem for uniformizations of Shimura varieties.
Recently, Eterovi\'{c} and Scanlon \cite{ES} have applied our result to the same problem for mixed period mappings.

Moreover, Ax-Schanuel results have significant applications in Diophantine geometry. 
For instance, Ax-Schanuel for period mappings \cite{BT} was used to prove Shafarevich type conjectures for hypersurfaces   \cite{LS}\cite{LV}; while Ax-Schanuel for mixed period mappings \cite{C} \cite{GK} was used in the higher dimensional Chabauty-Kim method \cite{H}. 
Gao used the Ax-Schanuel theorem for mixed Shimura varieties \cite{G} to study the generic rank of the Betti map \cite{Gao20}, which was then used by Dimitrov-Gao-Habegger \cite{DGH} to prove a uniform bound for the number of rational points on curves.
Ax-Schanuel also has applications on the geometric aspects of the Zilber-Pink conjecture \cite{BKU}\cite{BD}\cite{DR}\cite{PilSca} on unlikely intersections via the Pila-Zannier method \cite{PZ}.  A detailed account of the Zilber-Pink conjecture can be found in Pila's recent book \cite{Pila22}.

After the appearance of the initial version of this paper, Bakker and Tsimerman \cite{BT22} prove a geometric version of Andr\'{e}'s generalization of the Grothendieck period conjecture. This generalizes many Ax-Schanuel theorems.

We will use the notions of Mumford-Tate groups, variations of mixed $\ZZ$-Hodge structures, and mixed period mappings. A reference is  \cite{Kli17}.
 
\subsection{Statement of results}
Let $X$ be a smooth irreducible quasiprojective complex algebraic variety over $\CC$ equipped with an admissible graded-polarized variation $(\mc{H}, \mc{W}_{\bullet}, \mc{F}^{\bullet}, \mc{Q})$ of mixed $\ZZ$-Hodge structures (VMHS), where $\mc{H}$ is a local system of free $\mb{Z}$-modules of finite rank on $X$, $\mc{W}_{\bullet}$ is a weight filtration, $\mc{F}^{\bullet}$ is a Hodge filtration, and $\mc{Q}$ is a graded-polarization. 
Let $\eta$  be a Hodge generic point of $X$. Let $\Gamma$ be the monodromy group of the underlying local system of the variation.  Let $\tbG$ be the identity component of the  $\QQ$-Zariski closure of $\Gamma$ in the automorphism group of the underlying $\QQ$-vector space of the generic mixed Hodge structure at $\eta$.  Let $\tbG_u$ be the unipotent radical of $\tbG$. 
Let $H_\eta$ be the graded-polarized mixed Hodge structure at $\eta$. 
Let $\mc{M}$ be the classifying space parametrizing mixed $\RR$-Hodge structures with the same graded-polarization and Hodge numbers as $H_\eta$. 
Let $\widecheck{\mc{M}}$ be the corresponding projective space parametrizing decreasing filtrations \cite[\S 3.5]{BBKT}.
Let $D$ be the $\tbG(\RR)^+\tbG_u(\CC)$-orbit of $H_\eta$ in $\mc{M}$, where $\tbG(\RR)^+$ is the identity component of $\tbG(\RR)$.
Let $\widecheck{D}$ be the $\tbG(\CC)$-orbit of $H_\eta$ in  $\wc{\mc{M}}$.

Let $\tbG(\ZZ):= \tbG(\QQ)\cap \textbf{Aut}(\mc{H}_{\eta})$, where $\mc{H}_{\eta}$ is the stalk of $\mc{H}$ at $\eta$. 
Let $\tbG(\ZZ)^+:= \tbG(\mb{Z})\cap \tbG(\mb{R})^+$.
There exists a finite index subgroup $\Gamma_1$ of $\tbG(\ZZ)^+$ such that the quotient map $q: D\ra \Gamma_1\bs D$ is an unramified covering, see \cite[\S 3.3.1]{Kli17}.
First assume $\Gamma\subset \Gamma_1$. 
This is assumed everywhere outside Corollary \ref{Ax-Schanuel for variations of mixed Hodge structures: mod gamma}, Corollary \ref{mixed AS, coordinates} and their proofs.  Let $\psi : X \ra \Gamma\bs D$ be the period mapping\footnote{By \cite[Theorem A.5]{C}, the lifting of $\psi$ to the universal cover maps into $D$.} attached to the variation. 
Let $\phi$ be the composition of $\psi$ with $\Gamma\bs D\ra \Gamma_1\bs D$. 
Let $q':D\ra \Gamma\bs D$  be the quotient map.

\begin{definition}
Let $H$ be any mixed $\ZZ$-Hodge structure contained in $D$. 
Let $\tb{M}$ be a normal algebraic $\mb{Q}$-subgroup of the Mumford-Tate group $\tbMT_{H}$ of $H$. Let $\tb{M}_u$ be its unipotent radical. Let $\tb{M}(\RR)^+$ be the identity component of $\tb{M}(\RR)$. The $\tb{M}(\mb{R})^+\tb{M}_u(\mb{C})$-orbit $D(\tb{M})$ of $H$ is called a \emph{weak Mumford-Tate domain}. For any such $D(\tb{M})\subset D$, any irreducible component of  $\phi^{-1}q (D(\tb{M}))$ is called a \emph{weakly special subvariety} of $X$.
\end{definition}

By \cite[Corollary 6.7]{BBKT}, weakly special subvarieties are indeed algebraic.
Let $k$ be a non-negative integer. 
Let $J_k(X,\wc{D})$ be the set of all $k$-jets of germs of holomorphic mappings between open subsets of $X$ and $\wc{D}$. 
Since $X$ and $\wc{D}$ are algebraic, and since $X$ is smooth, $J_k(X,\wc{D})$ can be given an algebraic structure, see Lemma \ref{algebraicity of jet space}.
Let $\pi_X: J_k(X,\wc{D})\ra X$ be the mapping defined by projecting the $k$-jet of a germ to the center of the germ.
Let $W_k$ (resp. $W_{k,\Gamma}$) be the analytic subset of $J_k(X,\widecheck{D})$ consists of all $k$-jets of germs of  local liftings of the period mapping $\phi$ (resp. $\psi$).
For any irreducible analytic subset $U$ of $W_k$, denote by $U^{\Zar}$ its Zariski closure in $J_k(X,\wc{D})$.
The group $\tbG_{\CC}$ acts on $J_k(X, \wc{D})$ by postcomposition by automorphism.

\begin{thm}\label{main theorem}
 Let $U$ be an irreducible analytic subset of $W_k$. Let $s_k(U^{\Zar})$ be the generic dimension of the stabilizer of jets in $U^{\Zar}$.  If 
 $$\dim U^{\Zar}-\dim U<\dim \tbG(\CC)- s_k(U^{\Zar}),$$ 
 then $\pi_X(U)$ is contained in a proper weakly special subvariety of $X$.
\end{thm}

When $k=0$, we have $\dim \tbG(\CC)- s_k(U^{\Zar})=\dim \wc{D}$. This coincides with the term in \cite{C} and \cite{GK}.
In the case for Shimura varieties, it is known that $s_k(U^{\Zar})=0$ if $k\geq 2$ (\cite[Prop. 7.2]{MPT}), so the condition coincides with the one in \cite[Theorem 9.1]{MPT}.

The group $\Gamma_1$ is needed to construct definable fundamental set in Section \ref{Algebraic jet bundle attached to the mixed period mapping}, and thus we require $\Gamma\subset \Gamma_1$ in Theorem \ref{main theorem}.
Theorem \ref{main theorem} will be used to deduce the following corollary without the assumption that $\Gamma\subset \Gamma_1$.

\begin{cor}\label{Ax-Schanuel for variations of mixed Hodge structures: mod gamma}
Let $U$ be an irreducible analytic subset of $W_{k,\Gamma}$. Let $s_k(U^{\Zar})$ be the generic dimension of the stabilizer of jets in $U^{\Zar}$.   If 
$$\dim U^{\Zar}-\dim U<\dim \tbG(\CC)- s_k(U^{\Zar}),$$
then $\pi_X(U)$ is contained in a proper weakly special subvariety of $X$.
\end{cor}

Let $\Delta$ be the open unit disk. 
Let $d$ be a positive integer.
For any holomorphic map $g: \Delta^d\ra Z$ from the polydisk $\Delta^d$ to a complex manifold $Z$, 
we denote by $\rank(g)$ the generic rank of the map $g$.
We have the following version of mixed Ax-Schanuel in terms of transcendence degree and derivatives:

\begin{cor}\label{mixed AS, coordinates}
Let $\wt{\phi}$ be a local lifting of the period mapping $\phi$ on an open subset $B$. 
Let $v: \Delta^{\dim \wc{D}}\ra \widecheck{D}$ and  $u: \Delta^{\dim X}\ra  B$ be open embeddings, induced by algebraic coordinates on affine charts, such that  $(\wt{\phi}\circ u)(\Delta^{\dim X})\subset v(\Delta^{\dim \wc{D}})$. 
Let $f: \Delta^d\ra B$ be a holomorphic mapping such that $f(\Delta^d)\subset u(\Delta^{\dim X})$.
We have a map $\sigma:\Delta^d\ra W_{k,\Gamma}$ defined by $\sigma(z)=J_{k,f(z)}\wt{\phi}$, 
where $J_{k,f(z)}\wt{\phi}$ is the $k$-jet of $\wt{\phi}$ at $f(z)$.
Let $U$ be the image of $\sigma$. 
Let $s_k(U^{\Zar})$ be the generic dimension of the stabilizer of jets in $U^{\Zar}$. 
Write $z=(z_1,\dots, z_d)$, where $z_i$ are the coordinates of $\Delta^d$.
If
\begin{align*}
& \trd_{\CC} \CC((u^{-1}\circ f)(z), \partial^{\alpha}(v^{-1}\circ {\wt{\phi}}\circ u)((u^{-1}\circ f)(z)): |\alpha|\leq k)
\\&\quad< \rank (f)+\dim \tbG(\CC)- s_k(U^{\Zar}),
\end{align*}
then $f(\Delta^d)$ is contained in a proper weakly special subvariety of $X$.
\end{cor}

\subsection{Some recent literature on Ax-Schanuel for other functions}
There are several works proving Ax-Schanuel theorems for other functions. Baldi and Ullmo \cite{BU} prove the Ax-Schanuel theorem for certain non-arithmetic ball quotients. They use Simpson's theory in addition to o-minimality and monodromy (Andr\'{e}-Deligne). Bl\'{a}zquez-Sanz, Casale,  Freitag, and Nagloo \cite{BCFN} prove the Ax-Schanuel theorem with derivatives for uniformizers of any Fuchsian group of the first kind and any genus. Their proof uses Ax's original arguments, principal bundles with flat connections, the Maurer-Cartan structure equation, and the model theory of differentially closed fields. 
Nevanlinna theory approach to Ax-Schanuel theorems is also studied by Huang-Ng \cite{HN} and Noguchi \cite{Nog24}.
Papas \cite{Pap} proves the Ax-Schanuel theorem for the exponential functions for general linear groups.

\subsection{Strategy}
Bl\'{a}zquez-Sanz, Casale,  Freitag, and Nagloo established in  \cite{BCFN} the Ax-Schanuel theorem for algebraic principal bundles with flat connections. They proved that if the algebraic group acting on the bundle is sparse (a notion introduced in their paper concerning the analytic subgroups), and if the dimension of an algebraic subvariety of the bundle does not drop too much after intersection with a leaf, then the projection of the intersection under the bundle map is contained in a $\nabla$-special subvariety, which was also introduced in their paper. 

To use their result, we prove in Section \ref{Flat $G$-bundle attached to the mixed period mapping} that  when $k\gg 0$, the set  $P:=\tbG(\CC)\cdot W_k$ is an algebraic principal bundle over $X$, and that there is an algebraic flat connection on $P$ where each leaf is of the form $g\cdot W_{k,\Gamma}$ for some $g\in \tbG(\CC)$. In particular, the algebraicity is proved in Lemma \ref{algebraicity of the bundle} using Scanlon's work on algebraic differential equations from covering maps \cite{Sca18} and the definable fundamental set for the action of $\tbG(\ZZ)^+$ on $D$ constructed in \cite{C}. 
In Theorem \ref{construction, foliated principal bundle}, the algebraic flat connection is constructed using the definable GAGA \cite{BBT}.

In Section \ref{Ax-Schanuel for principal connections}, we use Andr\'{e}-Deligne \cite{A} to prove that any $\nabla$-special subvariety of $X$ is contained in a proper weakly special subvariety. We use the semisimple-unipotent Levi decomposition of $\tbG(\CC)$ in \cite{A} to prove that $\tbG(\CC)$ is sparse. Then in Section \ref{Proof}, we prove our main theorems for all $k\geq 0$ by applying the aforementioned Ax-Schanuel theorem for principal bundles \cite{BCFN} followed by projection to lower order jet spaces.

\subsection{Acknowledgements}
I would like to thank my advisor Jacob Tsimerman for helpful discussions and comments.  
I would also like to thank the referee for the helpful comments.

\section{Foliated jet bundle attached to the mixed period mapping}\label{Flat $G$-bundle attached to the mixed period mapping}\label{Algebraic jet bundle attached to the mixed period mapping}

A subset of $\RR^n$ is said to be definable if it is definable in the o-minimal structure $\RR_{an,\exp}$ \cite{DM}. We refer to \cite[Section 2]{JW} for an introduction to o-minimality.

The group $\tbG_{\CC}$ acts on $J_k(X, \wc{D})$ by postcomposition by automorphism.

By \cite[Section 5.1]{C}, there exists an open definable fundamental set $F'$ for the action of $\tbG(\ZZ)^+$ on $D$.
Since $\Gamma_1$ is of finite index in $\tbG(\ZZ)^+$, there exists an open definable fundamental set $F$ for the action of $\Gamma_1$ on $D$.
By \cite[Prop. 2.3]{BBKT}, $q|_F$ is definable. 
Let $W_{k,F}$ be the analytic set of all $k$-jets of germs of local liftings into $F$ of the period mapping $\phi$.

\begin{lemma}\label{definable graph}
The set $W_{k,F}$ is definable.
\end{lemma}

\begin{proof}
By \cite[Theorem 1.1]{BBKT} (see also \cite[Lemma 4.3]{C}), the period mapping $\phi$ is definable.
Since $J^d_k$ is a functor on the category of definable analytic spaces \cite[\S 4.6]{Urb23}, $J^d_kF$, 
$J^d_k (q|_F)$ and $J^d_k\phi$ are definable.
Therefore, $W_{k,F}$ is definable.
\end{proof}

\begin{lemma}\label{algebraicity of jet space}
Let $E$ be a field. Let $X$ and $Y$ be algebraic varieties over $E$. Suppose $X$ is smooth. 
The space $J_k(X,Y)$ of $k$-jets of germs of morphisms from $X$ to $Y$ can be equipped with an algebraic structure. 
For any non-negative integers $k$ and $\ell$, where $\ell \leq k$, we have an algebraic morphism $J_k(X,Y)\ra J_{\ell}(X,Y)$ lowering the order of jets.
\end{lemma}

\begin{proof}
By choosing affine open subsets $Y_1,\dots, Y_m$ of $Y$ and gluing $J_k(X, Y_i)$ together, we can assume that $Y$ is affine.
Let $X^{(k)}$ be the $k$-th infinitesimal neighborhood of $X$ for the diagonal morphism $X\ra X\times_E X$, see \cite[(16.7.1)]{Gro67}.
Define a functor by setting
\begin{align*}
J_k(X,Y)(A)
&= \Hom_A (X^{(k)}\times_X  A, Y\times_E A)
\end{align*}
for any scheme $A$ over $X$.
Since $X$ is smooth, the projection $X^{(k)}\ra X$ is finite and flat by \cite[Definition 16.10.1]{Gro67}.
Since $Y$ is affine, it is quasi-projective.
By \cite[p. 268]{Gro61}, the functor is then representable by the representability of $\underline{\Hom}_X(X^{(k)}, Y\times_E X)$.

The natural closed immersion $X^{(\ell)}\hookrightarrow X^{(k)}$ induces a natural transformation between the representable functors of $J_k(X,Y)$ and $J_{\ell}(X,Y)$. By Yoneda lemma, we then have the algebraic morphism $J_k(X,Y)\ra J_{\ell}(X,Y)$.
\end{proof}

A subset of an analytic variety is \emph{analytically constructible} (\emph{algebraically constructible}) if it is in the Boolean algebra generated by closed, complex analytic (algebraic) subvarieties.

\begin{lemma}\label{algebraicity of the bundle}
The set $P:=\tbG(\CC)\cdot W_k=\tbG(\CC)\cdot W_{k,\Gamma}$ is an algebraically constructible subvariety of $J_k(X, \wc{D})$.
\end{lemma}

\begin{proof}
Let $Y:=J^d_k(\Gamma_1\bs D)$.
Let $(J^d_kX)^*$ be the set of invertible $k$-jets into $X$ from the $d$-dimensional disk.
Let $Z_k:=(J^d_kX)^*\times_Y  J^d_kD$.
By the discussion after Lemma 2.18 in \cite{ES} (see also \cite{Sca18}), there is an algebraically constructible map $\widetilde{\chi}: J^d_k\wc{D}\ra Z$ to some algebraic variety $Z$ such that $\wt{\chi}(j_1)=\wt{\chi}(j_2)$ if and only if there exists $g\in \tbG(\CC)$ such that $g\cdot j_1=j_2$.
Thus the map $\chi:= \wt{\chi}\circ (J^d_kq)^{-1}\circ J^d_k\phi$ is well-defined and algebraically constructible (see \emph{loc. cit.} and Theorem 3.12 in \cite{Sca18}).
The set $\tbG(\CC) \cdot Z_k$ is therefore equal to the algebraic variety defined by the equation $\chi(t)=\wt{\chi}(y)$.
Define the algebraic morphism $c: (J^d_kX)^*\times J^d_k\wc{D}\ra J_k(X, \wc{D})$ by $c(u,v)=v\circ u^{-1}$.
The set $P=c(\tbG(\CC) \cdot Z_k)$ is then an algebraically constructible subvariety of $J_k(X, \wc{D})$.
\end{proof}

Let $K$ be the kernel of homomorphism $\tbG_{\CC}\ra \Aut(\wc{D})$ induced by the $\tbG_{\CC}$-action on $\wc{D}$. 
We explain the idea of the proof of the following lemma.  
We first prove that if $g\in \tbG(\CC)$ stabilizes the germ of a local lifting, then $g\in K$. By Noether's chain condition, a similar statement holds when the germ is truncated to some finite order jet. We then make this order independent of the local lifting using Lemma \ref{algebraicity of the bundle} and the chain condition a second time.

\begin{lemma}\label{free action on jet space}
There exists an integer $k_0>0$ such that  $K$ is the $\tbG_{\CC}$-stabilizer of any element in $W_k$ for any $k\geq k_0$.
\end{lemma}

\begin{proof}
Since $K$ is normal in $\tbG_{\CC}$, it suffices to show that there exists an integer $k_0>0$ such that $K$ is the $\tbG_{\CC}$-stabilizer of any element in $W_{k, F}$ for any $k\geq k_0$.
Let $x\in X$.
Let $\lambda: B\ra F$ be a local lifting of the period mapping into the fundamental domain $F$,
where $B$ is a definable open subset of $X^{an}$ containing $x$. 
Let $j_{k,x}\lambda$ be the $k$-jet of $\lambda$ at $x$. 
It is an operator from the $k$-jet space of $X$ at $x$ to the $k$-jet space of $D$ at $\lambda(x)$.
Let $S_{x,\lambda, k}$ be the $\tbG_{\CC}$-stabilizer of $j_{k,x}\lambda$.
Let $Z:=X\times_{\Gamma\bs\mc{M}}\mc{M}$.
Let $K'$ be the pointwise $\tbG(\CC)$-stabilizer of the image of $Z\ra \mc{M}$.
The map $\lambda$ is fixed by a conjugate $K''$ of $K'$ in $\tbG(\CC)$.
We have $K''=\bigcap_{k\geq  0}S_{x,\lambda, k}$ by the identity theorem. 

If $a\in K'$, then $a$ fixes the Zariski closure of the image of $Z\ra \mc{M}$ in $\wc{\mc{M}}$. By \cite[Lemma A.5]{C}, $D$ is contained in this Zariski closure, so $a$ fixes $D$.
Since $D$ is open in $\wc{D}$ and $\wc{D}$ is connected, $a$ fixes $\wc{D}$, so $a\in K$.
On the other hand,  if $a\in K$, then $a\in K'$ because the image $Z\ra \mc{M}$ is contained in $D$ by \cite[Lemma A.5]{C}.
Therefore, $K'=K$, so $K''=K$ by normality of $K$.

 The sequence $\{S_{x, \lambda, k}\}_{k\geq 0}$ of subgroups of $\tbG_{\CC}$ is decreasing. 
 Since $\tbG_{\CC}$ is Noetherian, there exists $k_{x,\lambda}>0$ such that $K=S_{x,\lambda, k}$ for all $k\geq k_{x,\lambda}$.
Since local liftings at $x$ are equal up to $\Gamma$-action, the group $S_{x,\lambda, k}$ is independent of $\lambda$ up to conjugation.
Since $K$ is normal, $S_{x,\lambda, k}$ is normal for all $k\geq k_{x,\lambda}$.
Hence, there exists $k_x>0$ such that for all $k\geq k_x$, the group $S_{x,\lambda, k}$ is equal to $K$ for all $\lambda$, and we can denote it by $S_{x,k}$.

Since $\lambda$ is definable, the operator $j_{k,x}\lambda$ is definable, and thus $S_{x,k}$ is definable.
Let $X_k:=\{x\in X: S_{x,k}=K\}$, which is a definable set.

The set $X_k$ is equal to the set of points in $X$ for which the $k$-jets of the germs of local liftings, centered at these points, into $F$ have $\tbG_{\CC}$-stabilizer equal to $K$.  
Let 
$$\alpha: \tbG(\CC)\times J_k(X, \wc{D})\ra J_k(X, \wc{D})\times J_k(X, \wc{D})$$ 
be defined by $(g, j)\mapsto (j, g\cdot j)$.
Let $\Delta$ be the diagonal in $J_k(X, \wc{D})\times J_k(X, \wc{D})$.
Let $A_k:=\Delta \cap \alpha((\tbG(\CC)\bs K)\times P)$.
The complement of the projection of $A_k$ in $X$ is equal to $X_k$.
By Lemma \ref{algebraicity of the bundle}, $P$ is algebraically constructible, so $A_k$ and thus $X_k$ are algebraically constructible.
From above discussion, $X=\bigcup_{k\geq 0}X_k$. The sequence $\{X_k\}$ is increasing.  Hence, there exists $k_0>0$ such that $X=X_k$ for all $k\geq k_0$. The lemma follows.
\end{proof}

\begin{thm}\label{construction, foliated principal bundle}
Let $k\geq k_0$. 
The map $\pi_{X}|_P: P\ra X$ is a principal $\tbG (\CC)/K$-bundle. 
It is equipped with an algebraic connection 
$$\nabla: TX\times_{X} P\ra TP,$$
i.e. it is an algebraic morphism satisfying the following properties:
\begin{itemize}
\item $d\pi_{X}|_P(\nabla_{v,p})=v$, where $\nabla_{v,p}$ means $\nabla(v,p)$,
\item ($\tbG_{\CC}$-invariance) $\nabla_{v, L_g(p)}=dL_g (\nabla_{v,p})$, where $L_g$ is the action on $P$ by $g$,
\item (flatness) the lift operator mapping any vector field $\mathcal{V}$ on $X$ to the vector field $\nabla_{\mathcal{V}}$ on $P$ given by $\nabla_{\mathcal{V}}(p)=\nabla(\mathcal{V}(\pi_{X}(p)),p)$ is a Lie algebra morphism.
\end{itemize}
There is a foliation on $P$ where each leaf is of the form $g\cdot W_{k,\Gamma}$ for some $g\in \tbG(\CC)$, and vice versa. The leaves are transverse to the fibers of the bundle.
\end{thm}

\begin{proof}
We have an algebraic morphism $(\tbG_{\CC}/K)\times P\ra P\times_{X}P$ given by $(gK,j,j')\mapsto (j, j', j, g\cdot j')$.
By Lemma \ref{free action on jet space}, it is an isomorphism, so $\pi_{X}|_P: P\ra X$ is a principal $\tbG (\CC)/K$-bundle.

There exists a cover $\{E_{\alpha}\}_{\alpha\in I}$ of $X$ by definable open subsets with definable local liftings 
$\lambda_{\alpha}: E_{\alpha}\ra D$ of period mapping, see \cite[Prop. 5.2]{BBKT}. 
Let $W_{k,\alpha}$ be the set of all $k$-jets of germs $\lambda_{\alpha}$.
By Lemma \ref{free action on jet space}, the map 
$$\kappa_{\alpha}:(\tbG(\CC)/K)\times E_{\alpha}\ra (\tbG(\CC)/K)\cdot W_{k,\alpha}=\pi_{X}|_P^{-1}(E_{\alpha})$$
defined by 
$(gK,x)\mapsto  g\cdot (J_{k,x}\lambda_{\alpha})$ 
is bijective,
where $J_{k,x}\lambda_{\alpha}$ is the $k$-jet of $\lambda_{\alpha}$ at $x$.
We define an algebraic connection $\nabla$ as follows: 
for any $(v,p)\in TE_{\alpha}\times_{X} P$, 
we let $\nabla_{v,p}=(d\kappa_{\alpha})_{\kappa_{\alpha}^{-1}(p)}(0,v)$.

Firstly, it is well-defined since it is independent of $E_{\alpha}$. Indeed, suppose  
$$(v,p)\in (TE_{\alpha}\times_{X} P)\cap  (TE_{\beta}\times_{X} P),$$
where $\alpha,\beta\in I$.
Write $$p=g_{\alpha}\cdot (J_{k,x}\lambda_{\alpha})=g_{\beta}\cdot (J_{k,x}\lambda_{\beta}),$$
where $g_{\alpha},g_{\beta}\in \tbG(\CC)$.
It suffices to show that 
$$(dL_{g_{\alpha}}\circ d(J_{k,x}\lambda_{\alpha})_x)(v)=(dL_{g_{\beta}}\circ d(J_{k,x}\lambda_{\beta})_x)(v).$$
There exists $\gamma\in \Gamma_1$ such that $\gamma\lambda_{\alpha}(x)=\lambda_{\beta}(x)$.
Thus, $\gamma\lambda_{\alpha}=\lambda_{\beta}$ on some small open subset in $E_{\alpha}\cap E_{\beta}$ containing $x$.
We thus have $J_{k,x}\lambda_{\alpha}=g_{\alpha}^{-1}g_{\beta}\gamma \cdot (J_{k,x}\lambda_{\alpha})$.
By Lemma \ref{free action on jet space}, $g_{\alpha}^{-1}g_{\beta}\gamma\in K$, 
so $L_{g_{\alpha}^{-1}g_{\beta}\gamma}\circ J_{k,x}\lambda_{\alpha}=J_{k,x}\lambda_{\alpha}$ for any $x\in E_{\alpha}$. 
Now we have
\begin{align*}
(dL_{g_{\alpha}}\circ d(J_{k,x}\lambda_{\alpha})_x)(v)
&=(dL_{g_{\alpha}}\circ dL_{g_{\alpha}^{-1}g_{\beta}\gamma} \circ d(J_{k,x}\lambda_{\alpha}))_x(v)
\\&= (dL_{g_{\beta}}\circ d(J_{k,x}(\gamma\lambda_{\alpha}))_x)(v)
\\&=(dL_{g_{\beta}}\circ d(J_{k,x}\lambda_{\beta})_x)(v).
\end{align*}

Since $\lambda_{\alpha}$ is definable, the map $\kappa_{\alpha}$ is definable.
Let $\Pi$ be the graph of the bundle map 
$$T(\tbG(\CC)/K)\times TE_{\alpha}\ra \tbG(\CC)/K\times E_{\alpha}.$$
Let $\Lambda$ the projection of $(\id,\id, \kappa_{\alpha})\Pi $ in $TE_{\alpha}\times (\tbG(\CC)/K)\cdot W_{k,\alpha}$.
Let $\eta_{\alpha}$ be the composition of $d\kappa_{\alpha}$ and the map
$$TE_{\alpha}\times (\tbG(\CC)/K)\cdot W_{k,\alpha}\ra T(\tbG(\CC)/K)\times TE_{\alpha}$$
given by $(v_1, p_1)\mapsto (0_{g_1}, v_1)$, where $g_1$ is the $\tbG(\CC)/K$-coordinate of ${\kappa^{-1}_{\alpha}(p_1)}$ and $0_{g_1}$ is the zero tangent vector at $g_1$.
The restriction of $\nabla$ to 
$\Lambda=TE_{\alpha}\times_{X} P$
is equal to the restriction of $\eta_{\alpha}$ to $\Lambda$.
Therefore, $\nabla$ is definable.
Let $s$ be an algebraic section of $TX\ra X$.
Then $\nabla(s)$ is a definable section of $TP\ra P$,
which induces a morphism $\mc{O}_P^{\defin}\ra (TP)^{\defin}$ of definable coherent sheaves defined by mapping $f\in \mc{O}_P^{\defin}(U)$ to $f\cdot \nabla(s)|_U$ for any definable open subset $U$ of $P$.
The section $\nabla(s)$ is thus algebraic by definable GAGA \cite[Theorem 3.1]{BBT}.
Therefore, $\nabla$ is algebraic.

Since $(\pi_{X}|_P\circ\kappa_{\alpha})(gK,x)=x$, we have $d\pi_{X}|_P(\nabla_{v,p})=v$.
Since $\kappa_{\alpha}$ is $\tbG_{\CC}$-equivariant, we have  $\nabla_{v, L_g(p)}=dL_g (\nabla_{v,p})$.
For any vector fields $\mathcal{V}, \mathcal{W}$ on $E_{\alpha}$,
we have 
$$\nabla_{[\mathcal{V},\mathcal{W}]}=(\kappa_{\alpha})_{*}[(0, \mathcal{V}), (0, \mathcal{W})]=[(\kappa_{\alpha})_{*}(0, \mathcal{V}),(\kappa_{\alpha})_{*}(0, \mathcal{W})]=[\nabla_{\mathcal{V}},\nabla_{\mathcal{W}}],$$
so $\nabla$ is flat. 

Let 
$$\mc{S}:=\{g\cdot W_{k,\alpha}: g\in \tbG(\CC), \alpha \in I\},$$
Define an equivalence relation $\sim$ on $\mc{S}$ as follows: $g_0\cdot W_{k,\alpha_0}\sim g_\ell\cdot W_{k,\alpha_\ell}$ in $\mc{S}$ if and only if there exist $g_i\cdot W_{k,\alpha_i}\in \mc{S}$ for each $0< i < \ell$, such that $g_{i-1}\cdot W_{k,\alpha_{i-1}}\cap  g_i\cdot W_{k,\alpha_i}\neq \varnothing$ for all $1\leq i\leq \ell$. Then we have a foliation on $P$ where each leaf has the same dimension as $X$ and is of the form
$$
\bigcup_{g\cdot W_{k,\alpha}\sim g_0\cdot W_{k,\alpha_0}} g\cdot W_{k,\alpha}
\quad \text{ for some }  g_0\cdot W_{k,\alpha_0}\in \mc{S},
$$
and vice versa.
The transversality follows from that $\kappa_{\alpha}$ is an isomorphism.
Hence each leaf is of the form $g_0\cdot W_{k,\Gamma}$ for some $g_0\in \tbG(\CC)$, and vice versa.  
\end{proof}

\section{Ax-Schanuel for principal bundles with flat connections}\label{Ax-Schanuel for principal connections}
We recall the definitions of $\nabla$-special subvarieties and sparse groups, and the Ax-Schanuel theorem for foliated principal bundles proved by Bl\'{a}zquez-Sanz, Casale,  Freitag, and Nagloo \cite{BCFN}.  Then we prove that any $\nabla$-special subvariety of $X$ is contained in a proper weakly special subvariety, and that  $\tbG(\CC)/K$ is sparse. 

Let $G$ be a complex algebraic group. Let $\nabla$ be a flat principal $G$-connection on a principal $G$-bundle $P$ over a complex algebraic variety $X$.  The Galois group $\Gal(\nabla)$ of $\nabla$ is the algebraic group $\{g\in G: g\cdot M=M\}$ for any minimal $\nabla$-invariant subvariety $M$ of $P$. A subvariety $Z$ of $X$ is \textbf{$\nabla$-special} \cite{BCFN} if for each irreducible component $Z_i$ with smooth locus $Z_i^*$,  the group $\Gal(\nabla |_{Z_i^*})$ is a proper subgroup of $G$.

A Lie subalgebra of the Lie algebra $\mf{g}$ of $G$  is said to be algebraic if it is the Lie algebra of an algebraic subgroup of $G$.
The algebraic envelop $\ol{\mf{h}}$ of a Lie subalgebra $\mf{h}$ of  $\mf{g}$ is the smallest algebraic Lie subalgebra containing $\mf{h}$.
An algebraic group $G$ is said to be \textbf{sparse} \cite{BCFN} if for any proper Lie subalgebra $\mf{h}\subset \mf{g}$, the algebraic envelop $\ol{\mf{h}}$ is a proper Lie subalgebra of $\mf{g}$.

\begin{thm}[\cite{BCFN}]\label{AS, connections}
Let $G$ be a sparse complex algebraic group. Let $\nabla$ be a flat principal $G$-connection on the principal $G$-bundle $P$ over a complex algebraic variety $X$. Assume that the Galois group $\Gal(\nabla)=G$. Let $V$ be an algebraic subvariety of $P$ and $L$ a horizontal leaf. Let $\mc{V}$ be an irreducible analytic subvariety of $V\cap L$. If $\dim V < \dim(\mc{V})+\dim G$, then the projection of $\mc{V}$ in $X$ is contained in a $\nabla$-special subvariety.
\end{thm}

\begin{thm}\label{del special, weakly special}
Let $P$ be the principal $\tbG_{\CC}/K$-bundle over $X$ with algebraic connection $\nabla$ in Theorem \ref{construction, foliated principal bundle}.  If $Z$ is an irreducible  $\nabla$-special subvariety in $X$, then $Z$ is contained in a proper weakly special subvariety. Moreover,  $\tbG_{\CC}/K=\Gal(\nabla)$.
\end{thm}

\begin{proof}
 Let $Z^*$ be the smooth locus of $Z$. 
 By Theorem \ref{construction, foliated principal bundle}, each horizontal leaf in $P$ is of the form $g\cdot W_{k,\Gamma}$ for some $g\in \tbG(\CC)$. 
Let $\ell$ be a horizontal leaf over $Z^*$, i.e. 
$\ell$ is a connected component of $\pi_{X}^{-1}(Z^*)\cap g\cdot W_{k,\Gamma}$ for some $g\in \tbG(\CC)$.
Let  $\Gamma'$ be the monodromy group for $Z^*$.
The group $g\Gamma'g^{-1}$ stabilizes $\ell$. Therefore, the algebraic group  $g\ol{\Gamma'}g^{-1}$ stabilizes the Zariski closure $\ol{\ell}$ of $\ell$. 
By definition and \cite[Lemma 2.2]{BCFN},  $\Gal(\nabla |_{Z^*})=\{gK\in \tbG_{\CC}/K: g\cdot \ol{\ell}=\ol{\ell}\}$, 
so $(g\ol{\Gamma'}g^{-1})K/K\subset \Gal(\nabla |_{Z^*})$. Then since $Z$ is $\nabla$-special, $(g\ol{\Gamma'}g^{-1})K/K$ is a proper subgroup of $\tbG_{\CC}/K$, so $\ol{\Gamma'}$ is a proper subgroup of $\tbG_{\CC}$.  By Andr\'{e}-Deligne (\cite{A}, \cite[Theorem A.8]{C}), $\pi_X(Z^*)$ is contained in a proper weakly special subvariety, and so is $\pi_X(Z)$ because $\pi_X(Z)$ and $\pi_X(Z^*)$ have the same closure and because weakly special subvarieties are closed \cite[Lemma 6.2]{BD}. 
Similarly,  $\ol{\Gamma}/K\subset \Gal(\nabla)$.
Since $\Gamma\subset \tbG(\ZZ)^+$, the $\QQ$-closure of $\Gamma$ is $\tbG$. By  \cite[Lemma A.4]{C}, the $\CC$-Zariski closure of $\Gamma$ is $\tbG_{\CC}$.  Therefore,  
$\tbG_{\CC}/K=\ol{\Gamma}/K\subset \Gal(\nabla),$
so $\tbG_{\CC}/K=\Gal(\nabla).$
\end{proof}

\begin{lemma}\label{semisimple unipotent sparse}
If $G$ is a linear algebraic group whose quotient by its unipotent radical $G_u$ is semisimple, then $G$ is sparse.
\end{lemma}

\begin{proof}
Let $\mf{g}, \mf{g}_s$ and $\mf{g}_u$ be the Lie algebras of $G, G_s$ and $G_u$ respectively, where $G_s$ is a Levi subgroup of $G$. Let $\mf{h}$ be a  Lie subalgebra of  $\mf{g}=\mf{g}_s\ltimes \mf{g}_u$. It is a general fact that  $\mf{h}$ is an ideal in its algebraic envelope $\ol{\mf{h}}$ \cite[Example 3.5]{BCFN}.  Suppose $\ol{\mf{h}}=\mf{g}$. By \cite[\S6, no. 8, Cor. 4]{B}, $\mf{h}\cap \mf{g}_u$ is the radical of $\mf{h}$ and $\mf{h}\cap \mf{g}_s$ is a Levi subalgebra of $\mf{h}$, so $\mf{h}=(\mf{h}\cap \mf{g}_s)\ltimes (\mf{h}\cap \mf{g}_u)$. The ideal $\mf{h}\cap \mf{g}_s$ of the semisimple Lie algebra $\mf{g}_s$ is semisimple, so there exists an algebraic subgroup $H_1$ of $G_s$ whose Lie algebra is $\mf{h}\cap \mf{g}_s$. Moreover, the exponential map gives an algebraic variety isomorphism \cite[Prop. 14.32]{Mil} between the unipotent group $G_u$ and its Lie algebra, so there exists an algebraic subgroup $H_2$ of $G_u$ whose Lie algebra is $\mf{h}\cap \mf{g}_u$. The Lie algebra of the algebraic subgroup $H_1\ltimes H_2$ of $G$ is thus $\mf{h}$, so $\mf{h}=\ol{\mf{h}}=\mf{g}$. Therefore, if $\mf{h}$ is proper, then $\ol{\mf{h}}$ is also proper.
\end{proof}

\begin{cor}\label{sparsity of monodromy}
The algebraic monodromy group $\tbG_{\CC}$ and the quotient $\tbG_{\CC}/K$ are sparse.
\end{cor}

\begin{proof}
By Andr\'{e} \cite[Corollary 2]{A}, the radical of $\tbG$ is unipotent, and thus it is equal to the unipotent radical.
Since quotient by radical is semisimple, $\tbG$ is a semidirect product of a semisimple group and the unipotent radical.
Formation of the unipotent radical commutes with field extensions in characteristic 0, so we can write  $\tbG_{\CC}=\tbG_s\ltimes \tbG_u$, where $\tbG_s$ is a semisimple Levi subgroup and $\tbG_u$ is the unipotent radical. By Lemma \ref{semisimple unipotent sparse}, $\tbG_{\CC}$ and $\tbG_{\CC}/K$ are sparse.
\end{proof}

\section{Proofs of main theorem and corollaries}
\label{Proof}

We now prove Theorem \ref{main theorem}, Corollary \ref{Ax-Schanuel for variations of mixed Hodge structures: mod gamma}, and Corollary \ref{mixed AS, coordinates}, which are restated as Theorem \ref{main theorem, restate}, Corollary \ref{Ax-Schanuel for variations of mixed Hodge structures: mod gamma, restate},  and Corollary \ref{mixed AS, coordinates, restate} below.

\begin{thm}\label{main theorem, restate}
 Let $U$ be an irreducible analytic subset of $W_k$. Let $s_k(U^{\Zar})$ be the generic dimension of the stabilizer of jets in $U^{\Zar}$.  If 
 $$\dim U^{\Zar}-\dim U<\dim \tbG(\CC)- s_k(U^{\Zar}),$$ 
 then $\pi_X(U)$ is contained in a proper weakly special subvariety of $X$.
\end{thm}

\begin{proof}
Recall that we assume $\Gamma\subset \Gamma_1$.
Let $S$ be the set of all distinct representatives of the cosets in $\Gamma_1/ \Gamma$.
We have 
$$W_k=\bigcup_{g\in S} g W_{k,\Gamma}$$ 
and 
$$\dim W_k=\dim X=\dim W_{k,\Gamma}.$$

Let $k_0$ be as in Lemma \ref{free action on jet space}.
First assume $k\geq k_0$.  
By Theorem \ref{construction, foliated principal bundle}, $P$ is a principal $\tbG(\CC)/K$-bundle over $X$. 
Since $U$ is irreducible, $g^{-1}U\subset W_{k,\Gamma}$ for some $g\in S$.
By Theorem \ref{construction, foliated principal bundle},  $W_{k,\Gamma}$ is a leaf in $P$.
We have $K\subset s_k(U^{\Zar})$ because $K$ fixes $\widecheck{D}$ pointwise and $U^{\Zar}\subset J_k(X,\widecheck{D})$.
Let $V$ be the smallest closed algebraic subvariety of $P$ containing $g^{-1}U$.
By Lemma \ref{algebraicity of the bundle}, $P$ is constructible in $J_k(X, \wc{D})$, so $V$ is also constructible in $J_k(X, \wc{D})$. 
Let $V^{\Zar}$ be the closure of $V$ in $J_k(X, \wc{D})$.
We then have 
$$\dim V=\dim V^{\Zar}=\dim (g^{-1}U)^{\Zar}=\dim U^{\Zar}.$$ 
Let $\mc{V}$ be the irreducible analytic subvariety of $V\cap W_{k,\Gamma}$ containing $g^{-1}U$. Then
\begin{align*}
\dim V
&=\dim U^{\Zar}
\\&<\dim U+\dim \tbG(\CC)- s_k(U^{\Zar})
\\&\leq \dim g^{-1}U+\dim \tbG(\CC)- \dim K
\\&\leq \dim \mc{V}+\dim(\tbG_{\CC}/K).
\end{align*}
We have 
$\pi_X(U)=\pi_X(g^{-1}U)\subset\pi_X(\mc{V}).$
By Corollary \ref{sparsity of monodromy}, Theorem \ref{AS, connections} and Theorem \ref{del special, weakly special},  $\pi_X(U)$ is contained in a proper weakly special subvariety of $X$. 

We now prove the theorem for $0\leq k<k_0$.
Let $P_{k_0}=\tbG(\CC)\cdot W_{k_0}$. 
Let $\rho: P_{k_0}\ra P$ be the restriction of the algebraic morphism 
$J_{k_0}(X, \wc{D})\ra J_{k}(X, \wc{D})$  
given by Lemma \ref{algebraicity of jet space} which lowers the order of jets.
By Lemma \ref{algebraicity of the bundle}, $\rho$ is algebraic.
Let $V_k$ be the smallest closed algebraic subvariety of $P$ containing $U$.
For any $j\in V_k$, and any $j_{k_0}$ in $\rho^{-1}(j)$, we have 
$$\rho^{-1}(j)=\{g\cdot j_{k_0}: g\cdot j=j\}\simeq \text{Stab}_{\tbG(\CC)}(j)/\text{Stab}_{\tbG(\CC)}(j_{k_0}).$$
Let $s_{k_0}$ be the generic dimension of stabilizer of jets in $W_{k_0}$.
We then have 
$$\dim \rho^{-1}(j)\leq \dim \text{Stab}_{\tbG(\CC)}(j)-s_{k_0}.$$
Since $V_k$ is constructible and $V_k^{\Zar}=U^{\Zar}$,  the generic dimension of the stabilizer of jets in $V_k$ is equal to $s_k(U^{\Zar})$. Hence,
$$\dim \rho^{-1}(V_k)\leq \dim V_k+ s_k(U^{\Zar})-s_{k_0}.$$
Let $U_{k_0}:=W_{k_0}\cap \rho^{-1}(U)$ and let $V_{k_0}$ be the smallest closed algebraic subvariety of $P$ containing $U_{k_0}$. 
We have $V_{k_0}\subset \rho^{-1}(V_k)$.    
Since jets of local liftings commutes with lowering order of jets,
the map $\rho|_{W_{k_0}}:W_{k_0}\ra W_k$ is onto,
so $\rho$ maps $U_{k_0}$ onto $U$. 
Since $\rho|_{W_{k_0}}$ has zero-dimensional fiber, we have $\dim W_{k_0}=\dim W_k$ and $\dim U_{k_0}=\dim U$.
We have 
$s_{k_0}\geq  \dim K= s_{k_0}(U_{k_0}^{\Zar})$
by Lemma \ref{construction, foliated principal bundle}.
We now have that
\begin{align*}
\dim U_{k_0}^{\Zar}-\dim U_{k_0}
&=\dim V_{k_0}-\dim U_{k_0}
\\&\leq \dim \rho^{-1}(V_k)-\dim U_{k_0}
\\& \leq  \dim V_k+ s_k(U^{\Zar})-s_{k_0}-\dim U_{k_0}
\\& =  \dim U^{\Zar}+ s_k(U^{\Zar})-s_{k_0}-\dim U_{k_0}
\\&< \dim \tbG(\CC)- s_k(U^{\Zar}) +\dim U +s_k(U^{\Zar})-s_{k_0}-\dim U_{k_0}
\\&\leq  \dim \tbG(\CC)-s_{k_0}(U_{k_0}^{\Zar}).
\end{align*}
By the case for $k=k_0$, $\pi_X(U_{k_0})$ is contained in a proper weakly special subvariety of $X$. We are done since $\pi_X(U)\subset \pi_X(U_{k_0})$.
\end{proof}

\begin{cor}\label{Ax-Schanuel for variations of mixed Hodge structures: mod gamma, restate}
Let $U$ be an irreducible analytic subset of $W_{k,\Gamma}$. Let $s_k(U^{\Zar})$ be the generic dimension of the stabilizer of jets in $U^{\Zar}$.   If 
$$\dim U^{\Zar}-\dim U<\dim \tbG(\CC)- s_k(U^{\Zar}),$$
then $\pi_X(U)$ is contained in a proper weakly special subvariety of $X$.
\end{cor}

\begin{proof}
Recall that we no longer assume $\Gamma\subset \Gamma_1$.
Let $\rho: \pi_1(X)\ra \Gamma$ be the monodromy representation attached to the VMHS. 
The group $\Gamma\cap \tbG(\ZZ)^+=\Gamma\cap \tbG(\ZZ)\cap \tbG(\RR)^+$ is of finite index in $\Gamma$ since $\ol{\Gamma}^{\Zar}/\tbG$ and $\tbG(\RR)/\tbG(\RR)^+$ are finite.
Hence, $\Gamma_0:=\Gamma\cap \Gamma_1$ is of finite index in $\Gamma$ since $\Gamma_1$ is of finite index in $\tbG(\ZZ)^+$.
Let $A=\rho^{-1}(\Gamma_0)$. Then $\pi_1(X)/A\cong \Gamma/\Gamma_0$ is finite. 
Let $f:\wt{X}\ra X$ be the finite covering such that $f_*(\pi_1(\wt{X}))=A$. The monodromy representation of the pullback of the VMHS to $\wt{X}$ is given by 
$\wt{\rho}: \pi_1(\wt{X})\ra \Gamma_0$. 
Since $\Gamma/ \Gamma_0$ and $\ol{\Gamma}^{\Zar}/\tbG$ are finite, we have 
$\dim \ol{\Gamma_0}^{\Zar}= \dim \ol{\Gamma}^{\Zar}=\dim \tbG$,
so $\ol{\Gamma_0}^{\Zar}=\tbG$ by connectedness of $\tbG$.
We have the period mapping $\wt{\phi}: \wt{X}\ra \Gamma_0\bs D$.
Let $J^{\nd,d}_k \wt{X}$ be the space of non-degenerate $k$-jets into $\wt{X}$ from the $d$-dimensional disk.
Consider the fiber product
\begin{center}
\begin{tikzcd}
W_{k,\Gamma_0}\arrow[r]\arrow[d] & J^d_kD\arrow[d,"J^d_kq"]
\\J^{\nd,d}_k\wt{X}\arrow[r, "J^d_k\wt{\phi}"] & J^d_k(\Gamma_0\backslash D).
\end{tikzcd}
\end{center}
The covering $\wt{X}\ra X$ induces an unramified covering $J^{\nd,d}_k\wt{X}\times J^d_k \wc{\mc{M}}\ra J^{\nd,d}_k X\times J^d_k \wc{\mc{M}}$.
Let $\wt{U}$ be an irreducible component of the preimage of $U$ under this map contained in $W_{k,\Gamma_0}$.
The stabilizer of any jet in $\wt{U}^{\Zar}$ is contained in the stabilizer of the image of the jet in $U^{\Zar}$, so $\dim s_k(\wt{U}^{\Zar})\leq \dim s_k(U^{\Zar})$.
We then have 
\begin{align*}
\dim \wt{U}^{\Zar}-\dim \wt{U}
&\leq \dim U^{\Zar}-\dim U
\\&<\dim \tbG(\CC)- s_k(U^{\Zar})
\\&\leq \dim \tbG(\CC)- s_k(\wt{U}^{\Zar}).
\end{align*}
By Theorem \ref{main theorem, restate}, the projection $\pi_{\wt{X}}(\wt{U})$ is contained in a proper weakly special subvariety of $\wt{X}$, so $\pi_X(U)$ is contained in a proper weakly special subvariety of $X$.
\end{proof}

\begin{cor}\label{mixed AS, coordinates, restate}
Let $\wt{\phi}$ be a local lifting of the period mapping $\phi$ on an open subset $B$. 
Let $v: \Delta^{\dim \wc{D}}\ra \widecheck{D}$ and  $u: \Delta^{\dim X}\ra  B$ be open embeddings, induced by algebraic coordinates on affine charts, such that  $(\wt{\phi}\circ u)(\Delta^{\dim X})\subset v(\Delta^{\dim \wc{D}})$. 
Let $f: \Delta^d\ra B$ be a holomorphic mapping such that $f(\Delta^d)\subset u(\Delta^{\dim X})$.
Write $z=(z_1,\dots, z_d)$, where $z_i$ are the coordinates of $\Delta^d$.
If
\begin{align*}
& \trd_{\CC} \CC((u^{-1}\circ f)(z), \partial^{\alpha}(v^{-1}\circ {\wt{\phi}}\circ u)((u^{-1}\circ f)(z)): |\alpha|\leq k)
\\&\quad< \rank (f)+\dim \tbG(\CC)- s_k(U^{\Zar}),
\end{align*}
then $f(\Delta^d)$ is contained in a proper weakly special subvariety of $X$.
\end{cor}

\begin{proof}
Using the coordinate charts $u$ and  $v$, the map $\sigma$ can be expressed as a tuple of functions consisting of $(u^{-1}\circ f)(z)$ and $\partial^{\alpha}(v^{-1}\circ {\wt{\phi}}\circ u)((u^{-1}\circ f)(z))$, where $|\alpha|\leq k$. 
We have 
$$\dim U^{\Zar}=\trd_{\CC} \CC((u^{-1}\circ f)(z), \partial^{\alpha}(v^{-1}\circ {\wt{\phi}}\circ u)((u^{-1}\circ f)(z)): |\alpha|\leq k).$$
We also have
$\rank(f)\leq \rank(\sigma)= \dim U$.
Then by assumption,
$$\dim U^{\Zar} <\dim U+\dim \tbG(\CC)- s_k(U^{\Zar}).$$  
so $\pi_X(U)$ is contained in a proper weakly special subvariety of $X$ by Corollary \ref{Ax-Schanuel for variations of mixed Hodge structures: mod gamma, restate}.
Since $f(\Delta^d)\subset \pi_X(U)$, the corollary follows. 
\end{proof}

\end{document}